\documentclass[smallcondensed,draft]{article}
\usepackage{amsfonts}
\usepackage{amsmath}
\usepackage{amsthm}
\usepackage{amsrefs}
\usepackage{longtable}
\usepackage{lscape}
\usepackage{comment}
\usepackage{url}
\usepackage{xcolor}

\usepackage[capitalise, noabbrev]{cleveref}
\crefname{ineq}{Inequality}{Inequalities}
\creflabelformat{ineq}{#2{\upshape(#1)}#3}

\newtheorem{theorem}{Theorem}[section]
\newtheorem{proposition}[theorem]{Proposition}
\newtheorem{lemma}[theorem]{Lemma}
\newtheorem{corollary}[theorem]{Corollary}

\theoremstyle{definition}
\newtheorem{definition}[theorem]{Definition}
\newtheorem{example}[theorem]{Example}

\newcommand{\F}{\mathbb{F}}
\newcommand{\Z}{\mathbb{Z}}
\DeclareMathOperator{\Tr}{Tr}
\DeclareMathOperator{\PG}{PG}
\DeclareMathOperator{\PGL}{PGL}
\DeclareMathOperator{\AG}{AG}
\DeclareMathOperator{\GL}{GL}

\newcommand{\orthogoval}{orthogoval}
\newcommand\st{\,:\,}

\DeclareMathOperator{\CA}{CA}

\DeclareMathOperator{\STS}{STS}
\DeclareMathOperator{\SQS}{SQS}
\DeclareMathOperator{\CPHF}{CPHF}
\DeclareMathOperator{\SCPHF}{SCPHF}
\DeclareMathOperator{\D}{D}

\def\Cr{\operatorname{Cr}}
\def\id{\operatorname{id}}

\usepackage{xspace}
\makeatletter
\DeclareRobustCommand\onedot{\futurelet\@let@token\@onedot}
\def\@onedot{\ifx\@let@token.\else.\null\fi\xspace}

\def\etal{{et al}\onedot}
\makeatother

\begin{document}
\date{October 18, 2022}

\title{Sets of mutually \orthogoval{} projective and affine planes}

\author{
Charles J. Colbourn\footnote{
School of Computing and Augmented Intelligence,
Arizona State University, Tempe, AZ 85287-8809, USA
{\tt Charles.Colbourn@asu.edu}} 
\and 
Colin Ingalls\footnote{
School of Mathematics and Statistics, Carleton University,
1125 Colonel By Drive, Ottawa, ON K1S 5B6, Canada
{\tt coliningalls@cunet.carleton.ca}; supported by NSERC}
\and 
Jonathan Jedwab\footnote{
Department of Mathematics, 
Simon Fraser University, 8888 University Drive, Burnaby, BC V5A 1S6, Canada, 
{\tt jed@sfu.ca}; supported by NSERC} 
\and
Mark Saaltink\footnote{unaffiliated researcher, {\tt mark@saaltink.ca}} 
\and 
Ken W. Smith\footnote{Department of Mathematics and Statistics, Sam Houston State University, Huntsville, TX 77341, USA
{\tt kenwsmith54@gmail.com}} 
\and
Brett Stevens\footnote{
School of Mathematics and Statistics, Carleton University, 
1125 Colonel By Drive, Ottawa ON K1S 5B6, Canada
{\tt brett@math.carleton.ca}; supported by NSERC}
}

\maketitle

\begin{abstract}
A pair of planes, both projective or both affine, of the same order and on the same pointset are {\em \orthogoval{}} if each line of one plane intersects each line of the other plane in at most two points.
In this paper we prove new constructions for sets of mutually \orthogoval{} planes, both projective and affine, and review known results that are equivalent to sets of more than two mutually \orthogoval{} planes.
We also discuss the connection between sets of mutually \orthogoval{} planes and covering arrays.
\end{abstract}

\section{Introduction}
In a projective plane of order $q$ an {\em oval} is a set of $q+1$ points, no three of which are collinear.  A beautiful theorem, independently published multiple times, states that in $\PG(2,q)$ there exists a set of $n^2+n+1$ ovals which form the blocks of a second projective plane \cites{baker_projective_1994,MR3248524,glynn_finite_1978,MR420430}. 
In the earliest proof, if $D$ is a $(n^2+n+1,n+1,1)$-difference set over $\Z_{n^2+n+1}$, then $-D$ is an oval in the projective plane developed from $D$. Since $-D$ is itself also a difference set, it can be developed into the second projective plane on the same ground set, $\Z_{n^2+n+1}$.  Hence we have constructed the blocks of two projective planes with the property that the lines of one are ovals in the other.  We introduce the term {\em \orthogoval{}} to describe this property.
\begin{definition}
 Two planes, both projective or both affine, of the same order and on the same pointset are {\em \orthogoval{}} if the blocks of one are ovals in the other, or in other words a line of one plane intersects any line in the other plane in at most two points.  A set of planes is called a {\em set of mutually \orthogoval{} planes} if the planes are pairwise \orthogoval{}.
\end{definition}
The primary goal of this paper is to prove new constructions for sets of mutually \orthogoval{} planes, both projective and affine, and to review known results that are equivalent to sets of more than two mutually \orthogoval{} planes.

A pair of \orthogoval{} projective planes of order $q$ was used by Raaphorst, Moura and Stevens to construct a family of covering arrays of strength 3, $\CA(2q^3-1;3,q^2+q+1,q)$ (see Definition ~\ref{def:CA}) from the circulant matrices of linear feedback shift register sequences over finite fields $\F_q$.  For $q > 3$ these are still the best known strength 3 covering arrays for their number of columns and alphabet sizes. Torres-Jimenez and Izquierdo-Marquez constructed covering arrays $\CA(2q^3-q;3,q^2-q+3,q)$ which are equivalent to deleting a particular set of $2q-2$ columns of the $\CA(2q^3-1;3,q^2+q+1,q)$ and then noting that, for every non-zero symbol $x$, the array that remains has two copies of the all-$x$ row.  Deleting one of each of these pairs completes the construction \cite{torres-jimenez_covering_2018}. Colbourn, Lanus and Sarkar found, via Sherwood covering perfect hash families, a $\CA(1016;3,8,64)$ using a conditional expectation search \cite{MR3770276}.  Investigating its structure they observed that if pairs of \orthogoval{} Desarguesian affine planes could be constructed, then keeping just the columns corresponding to the points of the affine planes from the $\CA(2q^3-1;3,q^2+q+1,q)$ and deleting one copy of any repeated row, would yield $\CA(2q^3-q;3,q^2,q)$ and thereby improve Torres-Jimenez and Izquierdo-Marquez's covering arrays by $q-3$ columns.  A second purpose of this paper is to discuss the connection between sets of mutually \orthogoval{} planes and covering arrays.

In a set of mutually \orthogoval{} planes, any three points collinear in one are non-collinear in all the others. Therefore, taking the collection of all lines from each plane in a set of $s$ mutually \orthogoval{} planes gives a strength 3 packing design \cite{colbourn_handbook_2007}.  The {\em packing number} $D(v,k,t)$ of a $v$-set is the maximum number of $k$-subsets (blocks) 
such that every $t$-subset appears in at most one block.  This can be used to calculate an upper bound on $s$.
\begin{theorem} \label{thm_upperbound}
  If there exists a set of $s$ mutually \orthogoval{}
projective planes of order $q$ then
\[
  s (q^2+q+1) \leq \D(q^2+q+1, q+1, 3).
  \]
  If there exists a set of $s$ mutually \orthogoval{}
affine planes of order $q$ then
  \[ s (q^2+q) \leq \D(q^2, q, 3).
  \]
  \end{theorem}
The Johnson bound \cite{colbourn_handbook_2007} on the size of packings is
  \[
    \D(v,k,3) \leq \left \lfloor \frac{v}{k} \left \lfloor \frac{v-1}{k-1} \left \lfloor \frac{v-2}{k-2} \right \rfloor\right \rfloor  \right \rfloor.
  \]
Since the blocks in $\AG(2,2)$ have size two, the set of $s$ copies of $\AG(2,2)$ is trivially mutually \orthogoval{} for every $s$ (similar to the fact that every set of $s$ latin squares of order 1 is mutually orthogonal). Otherwise, the Johnson bound gives the following result.
  \begin{corollary}\label{cor_upperbound}
    \begin{itemize}
      \item If there exists a set of $s$ mutually \orthogoval{}
projective planes of order $q$, then $s \leq \max\{5,q+2\}$.
      \item If there exists a set of $s$ mutually \orthogoval{}
affine planes of order $q > 2$, then $s \leq \max\{7,q+2\}$.
\end{itemize}
    \end{corollary}
We will either tighten or meet these upper bounds for some small values of~$q$.   

In \cref{sec_proj} we discuss projective planes.  There we give a new proof of the existence of a pair of \orthogoval{} projective planes from an algebraic geometric point of view using the Cremona transformation. We also consider if larger sets of such projective planes exist.  In \cref{sec_affine} we turn our attention to affine planes.  We prove that for even characteristic there always exist a pair of planes with the desired property.  When $\gcd(n,6)=1$ we construct a set of three mutually \orthogoval{} affine planes of order $2^n$. We then construct sets of size of seven for $q=4,8$. In \cref{sec_ca}, we discuss the connections between sets of mutually orthogoval{} planes and covering arrays.  We show that for affine planes we can extend the constructed arrays by two additional columns to construct a $\CA(2q^3-q;3,q^2+2,q)$, and demonstrate the utility of sets of more than two \orthogoval{} planes.  In \cref{sec_final} we give some final thoughts.

\section{Projective planes}\label{sec_proj}

\subsection{A pair of \orthogoval{} projective planes}
We start by giving a new proof of the fact that for every prime  power $q$ there exists a pair of Desarguesian \orthogoval{} $\PG(2,q)$.  This proof uses the Cremona transformation and its algebraic geometric approach adds another interesting point of view of this theorem.  For a general introduction to algebraic geometry and its terminologies, see \cite{MR0463157}.  While we believe that the result was broadly speaking known to algebraic geometers we do not know of a explicit articulation or proof of it in the literature.

We express the points of $\PG(2,q)$ with homogeneous coordinates using colons between the coordinates to make them distinct from other numbers that appear in round brackets in this article.  The {\em Cremona transformation}, or {\em standard quadratic transformation} is the partial rational function $\Cr : \PG(2,q)  \dashrightarrow \PG(2,q)$ defined by
\[
  \Cr((x:y:z)) = (x^{-1}: y^{-1}:z^{-1}).
\]
We can extend this partial map and will use the same symbol for this extension, 
\[
  \Cr((x:y:z)) = (xyzx^{-1}: xyzy^{-1}:xyzz^{-1}) = (yz:xz:xy).
\]
This map is undefined on the intersection points of the pairs of lines $yz$, $xz$ and $xy$ (variety $V(yz,xz,xy)$), $\{(1:0:0),(0:1:0),(0:0:1)\}$, and is birational since $\Cr\circ \Cr = \id$.  We are interested in the image of lines and conics under $\Cr$, which are the varieties of degree one and degree two homogeneous polynomials respectively.  For example, if $\ell$ is a line, we take $\Cr(\ell)$ to mean the set $\{\Cr(p):p \in \ell\}$.  This set might be a single point, a set of collinear points or a set of points lying on a conic.  It may not contain all the points of this line or conic because $\Cr$ is a partial map.  We can ``close'' the image set with the following procedure.  Let $C = V(f(x,y,z))$ in $\PG(2,q)$, not contained in the lines $V(x)$, $V(y)$, or $V(z)$.  Let $g(x,y,z) = f(1/x,1/y/1/z)x^a y^b z^c$, where $a$, $b$ and $c$ are minimal such that $g$ is a polynomial.  Then the closure of the image of $f$ under the Cremona transformation is  $\Cr(C) = V(g(x,y,z))$.  

The Cremona transformation has the following properties which are not hard to prove. See \cite{MR1042981,MR3100243} for examples of proofs of properties of $\Cr$.
\begin{proposition}
  \begin{itemize}
  \item $\Cr$ is bijective on $\PG(2,q) - V(xyz)$.
  \item The lines $x$, $y$ and $z$ map to points $(1:0:0)$, $(0:1:0)$ and $(0:0:1)$ respectively.
  \item The line $\beta y + \gamma z$ through $(1:0:0)$ maps to the line $\gamma y + \beta z$, also through $(1:0:0)$; and similarly for lines $\alpha x + \gamma z$ and $\alpha x + \beta y$ through $(0:1:0)$ and $(0:0:1)$ respectively.
  \item Each line $ax+by+cz$ that does not contain any of $(1:0:0)$, $(0:1:0)$ or $(0:0:1)$ maps to the irreducible conic $ayz+bxz+cxy$ which contains all three of $(1:0:0)$, $(0:1:0)$ and $(0:0:1)$.
  \end{itemize}
\end{proposition}

\begin{theorem} \label{thm_pair_proj}
For every prime power~$q$, there exists a pair of Desarguesian \orthogoval{} $\PG(2,q)$.
  \end{theorem}
  \begin{proof}
    Let $\omega \in \mathbb{F}_{q^3}$ so that $\omega \not\in \mathbb{F}_{q}$ and $\omega$ is not a root of $x^{3q^2} - 3x^{q^2+q+1} + x^{3q}+ x^3$.  Since 0 is a root with multiplicity 3 and 1 is also a root, if $q \geq 3$ then such an $\omega$ exists by the pigeonhole principle.  If $q=2$ then $\alpha+1$ satisfies these requirements where $\alpha$ is a root of $x^3+x+1$.  By these conditions, $(\omega:\omega^{q}:\omega^{q^2})$, $(\omega^{q}:\omega^{q^2}:\omega)$ and $(\omega^{q^2}:\omega:\omega^{q})$ are not collinear.  Let $\sigma$ be the projective linear transformation of $\PG(2, q^3)$ that maps $(1:0:0)$, $(0:1:0)$, $(0:0:1)$ and $(1:1:1)$ to  $(\omega:\omega^{q}:\omega^{q^2})$, $(\omega^{q}:\omega^{q^2}:\omega)$, $(\omega^{q^2}:\omega:\omega^{q})$ and $(1:1:1)$ respectively.  We conjugate the Cremona transformation by $\sigma$
    \[
      \Cr' = \sigma \circ \Cr \circ \sigma^{-1}
    \]
and consider its action on $\PG(2,q^3)$.

Since $Cr$ maps
\[
  \alpha (1:0:0) + \beta (0:1:0) + \gamma (0:0:1)
\]
to
\[
  \beta \gamma (1:0:0) + \alpha \gamma (0:1:0) + \alpha \beta (0:0:1),
\]
$\Cr'$ maps 
\[ \alpha  (\omega:\omega^{q}:\omega^{q^2})+ \beta (\omega^{q}:\omega^{q^2}:\omega) + \gamma (\omega^{q^2}:\omega:\omega^{q})
\]
to
\[
  \beta \gamma (\omega:\omega^{q}:\omega^{q^2}) + \alpha \gamma (\omega^{q}:\omega^{q^2}:\omega) + \alpha \beta (\omega^{q^2}:\omega:\omega^{q}).
\]
Lines not containing any of  $(\omega:\omega^{q}:\omega^{q^2})$, $(\omega^{q}:\omega^{q^2}:\omega)$, $(\omega^{q^2}:\omega:\omega^{q})$ are mapped to irreducible conics that contain all three of these conjugate points.

Suppose that $(x:y:z) \in \PG(2,q) \subset \PG(2,q^3)$.  Express this point as a sum $(x:y:z) = \alpha  (\omega:\omega^{q}:\omega^{q^2})+ \beta (\omega^{q}:\omega^{q^2}:\omega) + \gamma (\omega^{q^2}:\omega:\omega^{q})$.  Then the fact that the Frobenius map fixes this point ($(x:y:z)^{q} = (x:y:z)$) implies that
\begin{equation}
   \beta = \alpha^{q} \quad \mbox{ and }\quad  \gamma = \alpha^{q^2}. \label{coef_ident}
          \end{equation}
            Conversely if $\alpha$, $\beta$ and $\gamma$ satisfy \cref{coef_ident} then $\alpha  (\omega:\omega^{q}:\omega^{q^2})+ \beta (\omega^{q}:\omega^{q^2}:\omega) + \gamma (\omega^{q^2}:\omega:\omega^{q})$ is invariant under the Frobenius map and is therefore in the subplane $\PG(2,q)$.

Thus,
\begin{align*}
  \Cr'((x:y:z)) &= \beta \gamma (\omega:\omega^{q}:\omega^{q^2}) + \alpha \gamma (\omega^{q}:\omega^{q^2}:\omega) + \alpha \beta (\omega^{q^2}:\omega:\omega^{q}) \\
  &= \alpha^{q+q^2}  (\omega:\omega^{q}:\omega^{q^2}) + \alpha^{1+q^2} (\omega^{q}:\omega^{q^2}:\omega) + \alpha^{1+q} (\omega^{q^2}:\omega:\omega^{q}).
\end{align*}
The coefficients $\alpha^{q+q^2}$, $\alpha^{1+q^2}$ and $\alpha^{1+q}$ satisfy the identities in \cref{coef_ident} and so $\Cr'((x:y:z)) \in \PG(2,q)$ if and only if $(x:y:z) \in \PG(2,q)$.

Thus $\Cr'$ is a bijection on $PG(2,q)$ that maps lines to irreducible conics and its image forms a second plane which is \orthogoval{} to the first.
\end{proof}

\subsection{Sets of more than two \orthogoval{} projective planes}

A $\PG(2,2)$ is a Steiner triple system (STS) of order 7 and any 3 non-collinear points are an oval. Thus two $\STS(7)$ are \orthogoval{} if and only if their blocksets are disjoint.  It is known that there does not exist a set of three of $\STS(7)$s with mutually disjoint blocksets \cite{colbourn_triple_1999}.
When $q=3$ it is known that the collection of all 715 4-sets of a 13-set can be decomposed into 55 copies of $\PG(2,3)$ \cite{MR704245}. 
But not all of the lines in these copies will be ovals in the other copies: if a 4-set intersects a line in three points it is not an oval in the plane containing the line.  From \cref{cor_upperbound} we know that if a set of $s$ mutually \orthogoval{} projective planes in $\PG(2,3)$ exists then
\[
  13s \leq \D(13,4,3) = 65,
\]
and so $s \leq 5$. Chu and Colbourn have shown that a strength 3 packing of 4-sets on 13 points with a cyclic automorphism group has at most 52 blocks \cite{MR2059987}.  There does exist such a cyclic packing which can be decomposed into four projective planes of order 3.  This consists of the standard Desarguesian plane, together with the points of the conics
\begin{align*}
  z^2 + xy & =0\\
  z^2+yz+xy &= 0\\
  z^2+2yz+2x &=0
\end{align*}
and all the images of these conics in the Singer cycle.  In terms of difference sets the blocks of the Desarguesian plane are the development of $D = \{0,1,3,9\}$.  The three conics above correspond to difference sets $-D$, $2D$ and $-2D$ respectively.  Baker \etal proved that no conics are repeated in this set \cite{baker_projective_1994}, and it is easy to check that no triple is repeated in any block.  These give the blocks
\begin{align*}
  \{0,1,3,9\}+x,\quad \{0,1,5,11\}+x,\quad \{0,1,4,6\}+x,\quad \{0,1,8,10\}+x
  \end{align*}
  for $x \in \mathbb{Z}_{13}$.

 We now show that there does not exist a set of five \orthogoval{} projective planes of order~3. Suppose otherwise, for a contradiction. Then the set of all their blocks forms a maximum packing with $\D(13,4,3) =65$ blocks.  Each plane corresponds to a clique of size 13 in the 1-block intersection graph of the packing design.  In any such packing with 65 blocks, 26 triples are uncovered, the number of blocks containing a pair of points is 5, and each point is contained in exactly 20 blocks.  From this it can be deduced that the 26 missing triples are the blocks of a $\STS(13)$.  Adjoining one new point to each of these and adding them to the design yields a Steiner quadruple system on 14 points (a $\SQS(14)$).  Thus every $v=13$, $k=4$, $t=3$ packing with 65 blocks is obtained from an $\SQS(14)$ by the deletion of one point.  There are 4 non-isomorphic $\SQS(14)$ \cite{MR302474}.  Using SageMath \cite{sagemath}, the free open-source mathematics software system, for each of these $\SQS(14)$ we can delete each point and construct the 1-block intersection graph of the resulting packing.
None of these graphs has a clique of size larger than 9 (code available upon request), which gives the required contradiction.
  
To determine the maximum size of a set of mutually \orthogoval{} projective planes of orders $q=4$ and $q=5$ we built the 1-intersection graph of all the ovals in these two projective planes in SageMath.  Cliques of size $q^2+q+1$ correspond to projective planes with ovals for blocks.  A graph is then formed with all these possible cliques as vertices and adjacency determined by a pair being \orthogoval{}.  For both $q=4$ and $q=5$, no pairs of any of these planes were \orthogoval{},  thus there exists a pair of \orthogoval{} projective planes of orders $q=4$ and $5$ but no set of three. Source code is available upon request. 

We summarize the results of this subsection.
  \begin{theorem}\label{thm_proj_sum}
The size of a largest set of mutually \orthogoval{} projective planes in $\PG(2,q)$ is two for $q \in \{2,4,5\}$ and is four for $q=3$.
    \end{theorem}

\section{Affine planes}\label{sec_affine}

Because all lines have size two in an affine plane of order 2, a collection of any number of affine planes of order 2 on the same point set is trivially mutually \orthogoval{}.  If $q=3$ then the large set of $\STS(9)$ gives a collection of seven mutually \orthogoval{} affine planes in $\AG(2,3)$.  By \cref{cor_upperbound}, this is the most possible.  Up to isomorphism there are two distinct large sets of $\STS(9)$~\cite{MR2012539}.
  \begin{theorem}
     The size of a largest set of mutually \orthogoval{} affine planes in $\AG(2,3)$ is seven.
    \end{theorem}
This is the only case of \orthogoval{} affine planes of odd order which we know of, whereas when $q$ is even we can construct pairs or triples of \orthogoval{} affine planes and in some cases larger sets.  We identify the points of $\AG(2,q)$ as the subset of points $(x:y:1)$ from $\PG(2,q)$.
\subsection{A pair of \orthogoval{} affine planes}\label{sec_two_affine}

To show that there exists a pair of \orthogoval{} affine planes of order $q = 2^n$, we will need irreducible cubic polynomials in $\F_{q}[x]$ of the form $x^3 + bx + c$. When $n$ is even, $2^n-1$ is divisible by 3 and thus there exist elements $b \in \F_{2^n}$ which are not cube roots.  For each such $b$, the polynomial $x^3 - b$ is irreducible.  When $n$ is odd, every element is a cube root, $f(x) = x^3$ is a bijection and all polynomials of the form $x^3 + b$ are reducible.  However irreducible cubics of the desired form still exist. 
\begin{lemma}\label{cubic_lemma}
For each $n$, there exists an irreducible polynomial of the form $x^3+bx+c$ in $\F_{2^n}[x]$.
\end{lemma}
\begin{proof} There exist $q(q-1)$ monic quadratic polynomials $x^2+bx+c$ with $c\neq 0$, of which $(q-1)(q-2)/2 + (q-1) = q(q-1)/2$ are reducible and thus $q(q-1)/2$ are irreducible.  There are a total of $q(q-1)$ cubic polynomials $x^3 + bx + c$ with $c \neq 0$.  If such a cubic factors into linear terms, it must be of the form $(x-r_0)(x-r_1)(x-(r_0+r_1))$ with $r_0,r_1, r_0+r_1  \neq 0$.  Thus there are $(q-1)(q-2)/6$ such reducible cubics.  If such a cubic factors into an irreducible quadratic and a linear term, it must be of the form $(x^2+dx+e)(x+d)$ where $d,e \neq 0$ and the linear term is determined by the quadratic term.  Thus the number of these reducible cubics matches the number of irreducible quadratics, $q(q-1)/2$.  Thus there are $(q-1)(q+1)/3 > 0$ monic irreducible cubics with no $x^2$ term and non-zero constant term.
\end{proof}

For $q=2^n$, an oval $\mathcal{O}$ in $\PG(2,q)$ such that $(x_0+x_1:y_0+y_1:1) \in \mathcal{O}$ for all $(x_0:y_0:1), (x_1:y_1:1) \in \mathcal{O}$ is a {\em translation oval} \cite{hirschfeld_projective_1998}.  All ovals can be transformed in $\PGL(3,q)$ to an oval containing $(1:0:0)$, $(0:1:0)$, $(0:0:1)$ and $(1:1:1)$.  The translation ovals containing these four points are exactly $\{(x^{2^m}:x:1) : x \in \F_q\} \cup \{(1:0:0)\}$ with nucleus $(0:1:0)$ for all $\gcd(m,n)=1$.  We show that there always exists a pencil of conic translation ovals in $\PG(2,q)$. For this we need to characterize the conic translation ovals that may not contain $(1:0:0)$, $(0:1:0)$ or $(1:1:1)$.

\begin{lemma}\label{ovals_as_subspaces}
  If $q=2^n$, the conic oval in $\PG(2,q)$ defined by
  \[
    ax^2+by^2 +cz^2 + fyz + gxz +hxy= 0
  \]
  is a translation oval if and only if $h=c=0$.
\end{lemma}
\begin{proof}
  A translation oval must contain $(0:0:1)$ so $c=0$ is necessary.  Since the conic is an oval it contains a pair of points $(x_0:y_0:1)$ and $(x_1:y_1:1)$ such that $x_0y_1 \neq x_1y_0$.  Thus $(x_0+x_1:y_0+y_1:1)$ is also on the oval and
  \begin{align*}
    ax_0^2+by_0^2 + fy_0 + gx_0 +hx_0y_0 &= 0, \\
    ax_1^2+by_1^2 + fy_1 + gx_1 +hx_1y_1 &= 0, \\
    a(x_0+x_1)^2+b(y_0+y_1)^2 + f(y_0+y_1) + g(x_0+x_1) +h(x_0+x_1)(y_0+y_1) &= 0. \\
  \end{align*}
  The sum of these three equations is
  \[
    h(x_0y_1+x_1y_0) = 0
  \]
  and since $x_0y_1+x_1y_0 \neq 0$ we must have $h=0$.
  
 Conversely if $c=h=0$ and $(x_0:y_0:1)$ and $(x_1:y_1:1)$ are on the conic, then
  \begin{align*}
    ax_0^2 + by_0^2 + fy_0 + gx_0 &= 0, \\
    ax_1^2 + by_1^2 + fy_1 + gx_1 &= 0.
  \end{align*}
The quadratic  form evaluated at $(x_0+x_1:y_0+y_1:1)$ is
  \begin{align*}
    a(x_0+x_1)^2 &+ b(y_0+y_1)^2 + f(y_0+y_1) + g(x_0+x_1)  \\
    &= (ax_0^2 + by_0^2 + fy_0 + gx_0) + (ax_1^2 + by_1^2 + fy_1 + gx_1) &= 0
  \end{align*}
  so the conic is a translation oval.
\end{proof}
We will call two translation ovals which intersect only in $(0:0:1)$ {\em trivially intersecting}.

\begin{corollary}\label{oval_spread}
  For every $q=2^n$ there exists a trivially intersecting pencil of non-singular conic translation ovals in $\PG(2,q)$.
\end{corollary}
\begin{proof}
  Let $p = x^3 + bx + c$ be an irreducible cubic equation in $\mathbb{F}_q[x]$ (note $c \neq 0$). We note that $x^3 + b^2x + c^2$ is also irreducible.  Define conics
    \begin{align*}
      \phi &: x^2 + yz, \\
      \chi &: y^2 + byz + cxz.
    \end{align*}
    Both these satisfy \cref{ovals_as_subspaces} so they are translation ovals.  The conics determine the pencil  
    \[
      \{\alpha \phi + \beta \chi: (\alpha:\beta) \in \PG(1,q)\}.
      \]
      The partial derivatives of $\alpha \phi + \beta \chi$ are all zero at the point $(\alpha+\beta b:\beta c:0)$ and this is not on the conic because $x^3 + b^2x + c^2$ is irreducible and thus all conics in the pencil are non-singular.

      By the irreducibility of $p$, the equation $\phi(x,y,z) = \chi(x,y,z) = 0$ in $\PG(2,q)$ holds if and only if $(x:y:z) = (0:0:1)$.  The other three points of intersection are a conjugate triple in $\PG(2,q^3)$. Every conic in the pencil contains these four points, thus the pencil of conics $\{\alpha \phi + \beta \chi: (\alpha:\beta) \in \PG(1,q)\}$ meet inside $\PG(2,q)$ only in $(0:0:1)$ and are otherwise disjoint.
\end{proof}

In $\F_{2}^{2n}$, a set of $n$-dimensional spaces which pairwise intersect only in the origin is a {\em spread}.  The set of these spaces and their cosets form an affine plane of order $2^n$ which is called a {\em translation plane} because the group of all translations is a transitive automorphism group \cite{mr1434062}.  The Desarguesian affine plane of order $q=2^n$ can be constructed this way using the spread of 1-dimensional spaces of $\F_q^2$ which is isomorphic to  $\F_{2}^{2n}$ as a $\F_2$-vector space.  We call this the {\em line spread}.  These 1-dimensional spaces correspond to the lines $\alpha x + \beta y$ for $(\alpha:\beta) \in \PG(1,q)$.  They alternatively can be constructed from the cyclotomic cosets in $\F_{q^2}^*$.  If $\omega$ is a primitive element of $\F_q^2$, the set
\[
  C_0 = \langle \omega^{q+1} \rangle = \{\omega^{j(q+1)}: 0 \leq j < q-1\}
\]
is a multiplicative subgroup of $\F_{q^2}^*$.  Its cosets are 
\begin{equation}\label{line_spread}
  C_i = \omega^i \langle \omega^{q+1} \rangle
\end{equation}
for $0 \leq i \leq q$.  When we adjoin $0$ to each coset we get exactly the set of 1-dimensional subspaces of $\F_q^2$, a spread.  Non-Desarguesian translation planes can also be constructed from spreads \cite{mr1434062}.

We want to understand when the planes constructed from two spreads are \orthogoval{}.  Let $\mathcal{S} = \{S_i\}$ and $\mathcal{T} = \{T_j\}$ be two spreads in $\F_{2}^{2n}$.  If $x,y,z \in (u + S_i) \cap (v + T_j)$ then $0, y-x, z-x  \in S_i \cap T_j$ and conversely.  This gives the following result.
\begin{lemma}\label{fact:spreads}
  The affine planes constructed from $\mathcal{S}$ and $\mathcal{T}$ are \orthogoval{} if and only if  $|S_i \cap T_j| \leq 2$ for all $i$ and $j$.   
\end{lemma}

\cref{oval_spread} shows that there exists a ``spread'' of non-singular conics.  The translation plane constructed from this spread is Desarguesian.  

\begin{definition}
A map $f$ on $\AG(2,q)$ which is an $\F_2$-linear bijection and maps lines to ovals is {\em affine ovalinear}.
\end{definition}

\begin{lemma}\label{lemma_ovalinear}
If $f$ is an affine ovalinear map on $\AG(2,2^n)$ then the affine planes $\AG(2,2^n)$ and $\AG(2,2^n)^{f}$ are \orthogoval{} and both Desarguesian.
\end{lemma}
\begin{proof}
  Because $f$ is a bijection, the two planes $\AG(2,2^n)$ and $\AG(2,2^n)^{f}$ are isomorphic and therefore Desarguesian.  Since every line in $\AG(2,2^n)^{f}$ is an oval, the two planes are \orthogoval{}.
  \end{proof}
 We note that, by applying $f^{-1}$ to both planes, we have that $\AG(2,2^n)$ and $\AG(2,2^n)^{f}$ are \orthogoval{} if and only if $\AG(2,2^n)$ and $\AG(2,2^n)^{f^{-1}}$ are \orthogoval{}.  We will use \cref{lemma_ovalinear} to prove the next result and again in \cref{sec_three_affine}.

\begin{theorem} \label{pg_map}
The translation plane of order $q=2^n$ constructed from a trivially intersecting pencil of non-singular conics is Desarguesian and \orthogoval{} to the translation plane constructed from the line spread.
  \end{theorem}
  \begin{proof}
    Suppose that quadratic forms $\phi$ and $\chi$ generate a trivially intersecting pencil of translation conics.  Let $\psi$ be the quadratic form $z^2$.  Define the map on $\PG(2,q)$, $f((x:y:z)) = (\phi(x,y,z):\chi(x,y,z):\psi(x,y,z))$. Because $\phi$ and $\chi$ are translation conics, $f$ preserves the $z$-coordinate of all points in $\PG(2,q)$ and is thus $\F_2$-linear on $\AG(2,q)$.  Suppose that $(x:y:z)$ is on $\alpha \phi + \beta \chi + \gamma \psi$.  Then $f((x:y:z)) $ is on the line $\alpha x + \beta y + \gamma z$.  Since each point can be specified by the lines it intersects with, $f$ is a bijection which takes ovals to lines.  Thus $f^{-1}$ is affine ovalinear and so by \cref{lemma_ovalinear} $\AG(2,2^n)$ and $\AG(2,2^n)^{f}$ are \orthogoval{} and both Desarguesian.
\end{proof}

    \begin{corollary}\label{orthog_ag}
      There exists a pair of Desarguesian \orthogoval{} affine planes of order $q=2^n$.
    \end{corollary}
We note that the pair of $\PG(2,q)$ constructed, one from the lines $\alpha x + \beta y + \gamma z$ and the other from the set of conics $\alpha \phi + \beta \chi + \gamma \psi $, is \orthogoval{} {\em except} for the line $z$.  Thus any three points, not {\em all} from the line $z$, which are collinear in one $\PG(2,q)$ are non-collinear in the other.  We will use this fact when discussing the connections to covering arrays in~\cref{sec_ca}.

\subsection{Sets of three \orthogoval{} affine planes}\label{sec_three_affine}

For a restricted set of $n$, we can construct sets of three \orthogoval{} affine planes of order $q=2^n$. As in the construction of two \orthogoval{} affine planes in \cref{sec_two_affine}, we need to control the roots of specific polynomials.





\begin{lemma} \label{lemma:roots-4} If $k$ is relatively prime to $n$, then
  for any $a$ and $b$ in $\F_{2^n}$, the equation $ax^{2^k} + bx = 0$ has at most two roots in $\F_{2^n}$, unless $a=b=0$.
\end{lemma}
\begin{proof}
  $x=0$ is one root; any other must satisfy $ax^{2^k-1} + b = 0$. If $a=0$ this has no roots unless $b=0$.  For $a \not= 0$, the sole root is the $(2^k-1)$th root of $a^{-1}b$.  That root exists and is unique since $\gcd(2^n-1, 2^k-1) = 2^{\gcd(k,n)} - 1 = 1.$
\end{proof}

\begin{lemma} \label{lemma:roots-again}
  If $3k$ is relatively prime to $n$, then
the polynomial $x^{2^k+1} + x +1$ has no roots in $\F_{2^n}$.
\end{lemma}
\begin{proof}
  Suppose $r \in \F_{2^n}$ is a root of this polynomial and let $m=2^k$.  Then from
  \[ r^{m+1} = r +1 \] and taking the $m$th power (which is a Frobenius automorphism) of both sides we have
  \[ r^{m^2+m} = r^m + 1 \]
  so that
  \[ r^{m^2+m+1} = r^{m+1} + r = 1. \]
  Now taking the $(m-1)$th power gives $r^{m^3-1} = 1$, so that
  \[ r \in \F_{m^3} = \F_{2^{3k}}. \]
  But as we also have $r \in \F_{2^n}$ we conclude
  \[r \in \F_{2^{\gcd(3k,n)}} = \F_2. \]
  This is a contradiction, as neither 0 nor 1 is a root.
\end{proof}
A similar argument applies to a second polynomial we will encounter.
\begin{lemma} \label{lemma:more-roots}
  If $3k$ is relatively prime to $n$, then
the polynomial $x^{2^{2k}-1} + x^{2^k-1} +1$ has no roots in $\F_{2^n}$.
\end{lemma}
\begin{proof}
  Suppose $r \in \F_{2^n}$ is a root of this polynomial and let $m=2^k$.  Then from
  $r^{m^2-1} = r^{m-1} +1,$
  taking the $m$th power of both sides we have
  $r^{m^3-m} = r^{m^2-m} + 1.$
Multiplying by $r^{m-1}$ gives
  $r^{m^3-1} = r^{m^2-1} + r^{m-1} = 1.$
  So $r \in \F_{m^3}$, but as we also have $r \in \F_{2^n}$ we conclude
  $r \in \F_{2^{\gcd(3k,n)}} = \F_2.$
  This is again a contradiction, as neither 0 nor 1 is a root.
\end{proof}
Incidentally, this also shows (because it is $\F_2$-linear) that  $x^{2^{2k}} + x^{2^k} + x$ is a permutation polynomial for $\F_{2^n}$ when $n$ is relatively prime to $3k$.

Let $\phi_k(x,y) = (x^{2^k}+ y,\; y^{2^k}  + x + y)$ be a map defined on $\AG(2,2^n)$.  A simple expansion of terms shows
\begin{equation}\label{lemma:phi-step}
 \phi_{2k} = \rho \circ \phi_k\circ \phi_k \circ \rho
  \end{equation}
  where $\rho(x,y) = (y,x)$ is a linear transformation.

\begin{lemma} \label{lemma:phi-bijection}
  For all $k \geq 1$, if $3k$ is relatively prime to $n$, then $\phi_k$ is a bijection.
\end{lemma}
\begin{proof}
  Suppose $\phi_k(x,y) = \phi_k(s,t)$. Then
\begin{align*}
  x^{2^k} + y &= s^{2^k} + t \text{~and} \\
  y^{2^k} + x + y &= t^{2^k} + s + t
\end{align*}
giving
\begin{align*}
  (x+s)^{2^k}  &= y + t \text{~and} \\
  (y+t)^{2^k} &= (x+s) + (y+t).
\end{align*}
Substituting $y+t$ from the first equation into the second gives
\[ (x+s)^{2^{2k}} = (x+s) + (x+s)^{2^k}. \]
That is, $x+s$ is a root of $z^{2^{2k}} + z^{2^k} + z$. Then either $x=s$, or $x+s$ is
a root of $z^{2^{2k}-1}+z^{2^k-1}+1$, which \cref{lemma:more-roots} shows to be impossible.
\end{proof}

\begin{lemma} \label{lemma:phi-ortho}
If $3k$ is relatively prime to $n$,  then $\phi_k$ is affine ovalinear on $\F_{2^n}$.
\end{lemma}
\begin{proof}
It is clear that $\phi_k$ is $\F_2$-linear, and \cref{lemma:phi-bijection} shows it to be a bijection. We first find the images of lines through the origin.

The image of the line $x=0$ is the set
\begin{equation} \label{eq:line-0}
  \{ (y, y^{2^k}+y) \st y \in \F_{2^n} \}
\end{equation}
while the image of the line $y = ux$ for $u \in \F_{2^n}$ is
\begin{equation} \label{eq:line-non0}
  \{ (x^{2^k} + ux, u^{2^k} x^{2^k} + (u+1)x) \st x \in \F_{2^n} \}.
\end{equation}

Now we can look at the intersection of these sets with lines through the origin.
The line $x = 0$ meets the set~\eqref{eq:line-0} in the set
indexed by $\{ y \in \F \st y = 0 \}$, which has one member.  It meets the set~\eqref{eq:line-non0}
in the set indexed by
\[ \{ x \in \F \st x^{2^k} + ux  = 0 \} \]
which \cref{lemma:roots-4} shows to have at most 2 members.

For $w \in \F$, the line $y = wx$ meets the set~\eqref{eq:line-0} in the set
indexed by
\[ \{ y \in \F \st y^{2^k} + (w+1)y = 0 \}, \]
which \cref{lemma:roots-4} shows to have at most 2 members.
It meets the set~\eqref{eq:line-non0} in the set indexed by
\[ \{ x \in \F \st (u^{2^k}+w)x^{2^k} + (uw + u+1)x = 0 \}. \]
\cref{lemma:roots-4} applies again unless
\[ u^{2^k} + w = 0 = uw + u+1 \]
but then $w = u^{2^k}$ and $u^{2^k+1} + u + 1 = 0$. 
So $u$ is a root of $x^{2^k+1} + x + 1$, 
which \cref{lemma:roots-again} shows to be impossible.
\end{proof}

\begin{theorem}
  If $6k$ and $n$ are relatively prime and $q=2^n$,
  the standard plane $\AG(2,q)$, its image $\AG(2,q)^{\phi_k}$, and $\AG(2,q)^{\phi_k \circ \phi_k}$ are mutually \orthogoval{}.
\end{theorem}
\begin{proof}
   By \cref{lemma_ovalinear,lemma:phi-ortho},
   we have that $\AG(2,q)$ and  $\AG(2,q)^{\phi_k}$ are \orthogoval{}.
  Applying bijection $\phi_k$ to both of these planes shows that $\AG(2,q)^{\phi_k}$ and  $\AG(2,q)^{\phi_k \circ \phi_k}$ are also \orthogoval{}.
  It remains to show that the standard plane and  $\AG(2,q)^{\phi_k \circ \phi_k}$ are \orthogoval{}.
  \cref{lemma_ovalinear,lemma:phi-ortho} show that
  $\AG(2,q)$ is \orthogoval{} to $\AG(2,q)^{\phi_{2k}} = \AG(2,q)^{\rho \circ \phi_{k} \circ \phi_k \circ \rho}$.
  Since $\rho$ is a bijection
  and $\rho \circ \rho$ is the identity, we can first apply $\rho$ to those two planes to show that $\AG(2,q)^\rho$ is \orthogoval{} to $\AG(2,q)^{\rho \circ \phi_{k} \circ \phi_k}$, and then, as $\AG(2,q)^\rho = \AG(2,q)$,
  we have $\AG(2,q)$ \orthogoval{} to $\AG(2,q)^{\phi_k \circ \phi_k}$.
\end{proof}

Clearly $n$ must always be relatively prime to 6 for this construction.  For $k \in \{ 2^i3^j | i, j \geq 0\} = \{1,2,3,4,6,8,\dotsc\}$ this is sufficient.

\begin{corollary} \label{cor_3_affine_planes}
  If $n$ is relatively prime to 6, there exists a set of three mutually \orthogoval{} affine planes of order $2^n$.
\end{corollary}
If $n >3$ is prime then every $k$ satisfying $1 \leq k \leq n-1$ gives a triple of mutually \orthogoval{} affine planes.
We do not know whether or not the sets for different $k$ are isomorphic.

\subsection{Sets of more than three \orthogoval{} affine planes}
We can use the construction of translation affine planes from spreads to construct sets of more than three \orthogoval{} affine planes for $q=4$ and $8$.  The line spread $\mathcal{C}$ defined just after \cref{line_spread} constructs the standard Desarguesian affine plane on pointset $\F_{2^{2n}} \cong \F_{2}^{2n}$.  Other spreads which construct isomorphic translation planes can be generated by applying elements $M \in \GL(2n,\F_2)$. By \cref{fact:spreads}, if the intersections of the spaces in the image spread with the spaces from the original line spread all have size no more than two then the pair of affine planes constructed is \orthogoval{}.  Furthermore if this is true for $M^i$ for each $i$ satisfying $1 \leq i < s$ then the $s$ spreads $\{M^i(\mathcal{C}): 0\leq i < s\}$ generate a set of $s$ mutually \orthogoval{} affine planes. 

This gives an efficient method for computationally searching for a set of mutually \orthogoval{} affine planes of $\AG(2,2^n)$.  
Generate matrices $M$ from $\GL(2n,\F_2)$, retaining those for which the plane constructed from $M(\mathcal{C})$ is \orthogoval{} to that built from $\mathcal{C}$.  Given a large set $M_1, M_2, \dots, M_r$ of such matrices, construct a graph with vertex set $M_0 = I_{2n}, M_1, M_2, \dots, M_r$ in which vertices $M_i$ and $M_j$ are joined when $M_i(\mathcal{C})$ is \orthogoval{} to~$M_j(\mathcal{C})$. A clique of size $c$ in this graph is then a set of $c$ mutually \orthogoval{} affine planes in $\AG(2,2^n)$. This search yielded a set of seven mutually \orthogoval{} affine planes of order~$4$ (\cref{ex:q4sevenoo}) and a set of seven mutually \orthogoval{} affine planes of order~$8$ (\cref{ex:q8sevenoo}).

\begin{example}
  \label{ex:q4sevenoo}
Let $q=4$ and let $\omega$ be a root of the primitive polynomial $x^4+x+1$.
The cyclotomic cosets $C_i$ of $\F_{16}^*$ and the corresponding additive subgroups $S_i$ of $\F_2^4  \cong \F_{16}$ (using the $\F_2$-linear map $(t_0 t_1 t_2 t_3) \rightarrow t_0 \omega^3 + t_1 \omega^2 + t_2 \omega + t_3$) are given by
\[
\begin{array}{lll}
i & \mbox{$C_i$ (multiplicative)}	& \mbox{$S_i$ (additive)} 
								\\ \hline
0 & \{1,     \omega^5,\omega^{10} \}	& \{0000, 0001, 0110, 0111\}	\\
1 & \{\omega,\omega^6,\omega^{11} \}	& \{0000, 0010, 1100, 1110\}	\\
2 & \{\omega^2,\omega^7,\omega^{12} \}	& \{0000, 0100, 1011, 1111\}	\\
3 & \{\omega^3,\omega^8,\omega^{13} \}	& \{0000, 1000, 0101, 1101\}	\\
4 & \{\omega^4,\omega^9,\omega^{14} \}	& \{0000, 0011, 1010, 1001\}.
\end{array}
\]
The 4 lines of the parallel class containing $S_0$ are the cosets
\begin{align*}
L_{01} = 0000 + S_0 &= \{0000, 0001, 0110, 0111\}	\\
L_{02} = 0010 + S_0 &= \{0010, 0011, 0100, 0101\}	\\
L_{03} = 1000 + S_0 &= \{1000, 1001, 1110, 1111\}	\\
L_{04} = 1010 + S_0 &= \{1010, 1011, 1100, 1101\}.	
\end{align*}
The remaining 16 lines of the Desarguesian affine translation plane $\AG(2,4)$ are formed similarly from $S_1, \dots, S_4$.

Let
\[
M_{4} = 
\begin{bmatrix}
0 & 1 & 1 & 0 \\
0 & 0 & 0 & 1 \\
1 & 1 & 0 & 0 \\
0 & 0 & 1 & 1
\end{bmatrix}.
\]
Then $\{ M_4^{i}(\AG(2,4)): 0 \leq i < 7\}$ is a set of seven mutually \orthogoval{} affine planes of order~$4$.
\end{example}

With $\F_4$ generated by $x^2+x+1$ and primitive element $\zeta = \omega^5$, the planes $M_4^i(\AG(2,4))$ for $1 \leq i < 7$ correspond to the six pencil spreads of conics formed from the pairs
\begin{align*}
      \phi_1 &: x^2 + (\zeta+1) yz + \zeta xz &\chi_1 &:  y^2 + yz + xz;\\
  \phi_2 &: x^2 + (\zeta+1)yz + xz &\chi_2 &:  y^2 +  (\zeta+1)xz;\\
  \phi_3 &: x^2 +  xz +\zeta xy &\chi_3 &:  y^2 +  \zeta yz + \zeta xz; \\
  \phi_4 &: x^2 +  \zeta yz +\zeta xz &\chi_4 &:  y^2 +  yz + (\zeta+1) xz; \\
  \phi_5 &: x^2 +  yz + xz &\chi_5 &:  y^2 +  yz + \zeta  xz; \\
  \phi_6 &: x^2 +  \zeta yz  &\chi_6 &:  y^2 +  xz.\\  
\end{align*}
The union of the lines from all seven of the planes forms a $\SQS(16)$.  A decomposition of an $\SQS(16)$ into affine planes was previously known \cite{mathon_searching_1997} but this is an interesting method of construction.  It is unknown if the 1820 4-subsets of a 16-set can be partitioned into $\SQS(16)$s  \cite{MR4113538}; this would be a large set of Steiner Quadruple Systems.  Because the cyclotomic cosets $C_i$ are all disjoint, the collection of all cyclotomic cosets over these seven spreads forms the blocks of a resolvable Steiner triple system of order 15, a solution to Kirkman's schoolgirl problem of 1850 \cite{kirkman_query_1850}:
\begin{quote}
``Fifteen young ladies in a school walk out three abreast for seven days in succession: it is required to arrange them daily, so that no two shall walk twice abreast.''
\end{quote}
This design is the derived design of the $\SQS(16)$ at the point $0$. Because the planes have a transitive automorphism group, deriving at any other point gives an isomorphic resolvable $\STS(15)$.

\begin{example}
  \label{ex:q8sevenoo}
For $q=8$, let $\F_{64}$ be generated by a primitive element with minimum polynomial $x^6+x+1$.  Construct the Desarguesian affine translation plane $\AG(2,8)$.  Let
\[
M_{6} = 
\begin{bmatrix}
0 & 1 & 0 & 1 & 0 & 1 \\
0 & 0 & 1 & 0 & 1 & 1 \\
1 & 1 & 0 & 0 & 0 & 0 \\
0 & 0 & 1 & 1 & 1 & 1 \\
1 & 1 & 1 & 0 & 0 & 1 \\
0 & 0 & 1 & 1 & 1 & 0
\end{bmatrix}.
\]
Then $\{ M_6^{i}(\AG(2,8)): 0 \leq i < 7\}$ is a set of seven mutually \orthogoval{} affine planes of order~$8$.
\end{example}
Only three of these affine planes, those for $i=1,2,4$, are formed from a pencil spread of conics.  With $\F_8$ generated from $x^3 + x^2 + 1$ and $\omega$ as its root, the pencils of conics are formed from
\begin{align*}
      \phi_1 &: x^2 + \omega yz + \omega xz &\chi_1 &:  y^2 + (\omega^2+\omega+1)yz + (\omega+1)xz;\\
  \phi_2 &: x^2 + (\omega^2+\omega)yz + xz &\chi_2 &:  y^2 + yz + (\omega^2+1) xz;\\
  \phi_4 &: x^2 + (\omega^2)yz + (\omega+1) xz +\alpha xy &\chi_4 &:  y^2 +  (\omega^2+\omega)yz + (\omega^2+\omega) xz.
\end{align*}
The spreads for $i = 3,5,6$ are spreads of conics but are not pencils.

The procedure to find suitable matrices $M \in \GL(2n,\F_2)$ was implemented in~$C$, and the clique-finding algorithm was implemented in \emph{Mathematica}.
The random search was modified in order to allow the C program to run in reasonable time for $q=16$. Rather than seeking a suitable matrix $M$ directly, we choose the first $2n-1$ rows of $M$ randomly and calculate its product with the first $2n-1$ coordinates of each line~$S_i$. If no four such products from the nonzero points of a line $S_i$ occur in the set of products from a single line $MS_{i'}$, then we proceed to test all possible final rows for~$M$; otherwise we discard the partial matrix~$M$.

For $q=8$, the C program running on a 2013 iMac desktop takes on average less than 0.2 milliseconds to find each suitable matrix $M$, and a set of $1200$ such matrices was sufficient to find the first clique of size 7.  We later refined it to the solution given in \cref{ex:q8sevenoo}. No clique of size larger than 7 was found from a set of 50,000 such matrices.

For $q=16$, the C program takes on average 15 seconds to find each suitable matrix~$M$. A set of $1,000,000$ such matrices was collected over a number of months on several machines. This set is too large for the clique-finding algorithm of \emph{Mathematica} to handle, but it is straightforward to search in C for a pair of non-identity matrices $M$, $M'$ from the set for which $M(\AG(2,6))$ is \orthogoval{} to $M'(\AG(2,6))$; together with the identity matrix $I_8$, this would give a set of three mutually \orthogoval{} affine planes of order~$16$. Unfortunately, no such pair was found.  The set of $1,000,000$ suitable matrices does not contain an example $M$ for which $M^2$ is also suitable. It remains an open question as to whether there exists a set of more than two mutually \orthogoval{} affine planes of orders $2^n$ with $\gcd(n,6) > 1$ and $n>3$.

We summarize the results of this subsection, using \cref{cor_upperbound}.
  \begin{theorem} \label{thm_affine_sum}
    \begin{itemize}
      \item The size of a largest set of mutually \orthogoval{} affine planes in $\AG(2,4)$ is seven.
      \item There exists a set of seven mutually \orthogoval{} affine planes in $\AG(2,8)$.
      \end{itemize}
    \end{theorem}

\section{Connections to covering arrays}\label{sec_ca}
Here we discuss the connection between sets of $s$ mutually \orthogoval{} planes, covering perfect hash families and covering arrays.  

\begin{definition} A {\em covering perfect hash family} $\CPHF_{\lambda}(n;t,k,q)$ is a $n \times k$ array of elements from $\mathbb{F}_q^t\setminus \vec{0}$ (equivalently the points of $\PG(t-1,q)$) such that, for each set $T$ of $t$ columns, there exist at least $\lambda$ rows whose entries in the columns of $T$ are linearly independent.  If the vector entries all have non-zero last coordinate (equivalently, they are the points of the affine subspace of $\PG(t-1,q)$) then the $\CPHF$ is a {\em Sherwood covering perfect hash family} $\SCPHF_{\lambda}(n;t,k,q)$.  
\end{definition}

\begin{definition} \label{def:CA} A {\em covering array} $\CA_{\lambda}(N;t,k,v)$ is a $N \times k$ array of elements from a $v$-set such that, for each set $T$ of $t$ columns and each $t$-tuple from the $v$-set, there exist at least $\lambda$ rows whose entries in the columns of $T$ match the $t$-tuple.  
\end{definition}

Covering perfect hash families are used to construct covering arrays. We extend the standard constructions \cite{MR3770276,MR2214495} by considering the index of the constructed covering array.
\begin{theorem} 
  Suppose that $C$ is a $\CPHF_{\lambda}(n; t, k,q)$ and $\lambda' \leq \lambda$. Then there exists a $\CA_{\lambda'}(n(q^t-1) + \lambda' ; t,k, q)$; and
there exists a $\CA_{\lambda'}(n(q^t-q)+\lambda'q; t, k , q)$ if $C$ is an $\SCPHF(n; t,k, q)$.
\end{theorem}
\begin{proof}
  See \cite{MR3770276,MR2214495} for the basic proof.  The extension to higher index is straightforward.  For a given set $T$ of $t$ columns, {\em any} row in which the entries from $T$ are linearly independent contributes one copy of every $t$-tuple over $\mathbb{F}_q$ to the covering array, so every $t$-tuple appears at least $\lambda$ times. Every row of a  $\CPHF$ contributes an all-zero row to the covering array.  Every row of a $\SCPHF$ contributes each of the $q$ constant rows to the covering array.  In each case all but $\lambda'$ can be deleted. 
\end{proof}

\begin{theorem}\label{thm_orthog2cphf}
  Suppose that a set of $s$ mutually \orthogoval{} Desarguesian projective (affine) planes exists.  Then there exists a $\CPHF_{s-1}(s;3,q^2+q+1,q)$ ($\SCPHF_{s-1}(s;3,q^2,q)$). 
\end{theorem}
\begin{proof}
  Represent the points of $\PG(2,q)$ as length $3$ vectors in homogeneous coordinates over $\mathbb{F}_q$ and the points of $\AG(2,q)$ as length $3$ vectors in homogeneous coordinates over $\mathbb{F}_q$ with last entry $1$. For each $i$ satisfying $1 \leq i \leq s$, let $\phi_i$ be an isomorphism of the $i$th plane in the set onto the first plane.

 Let $C$ be the array with $s$ rows and columns indexed by the elements of the common point set of the planes.  Define the entries of $C$ by 
  \[
    C_{iz} = \phi_i(z).
  \]
  Let $T$ be a set of three columns corresponding to points $z_1$, $z_2$ and $z_3$. When these three points are non-collinear in the plane $i$, the corresponding entries in row $i$ are linearly independent.  If $z_1$, $z_2$ and $z_3$ appear on a line in one of the planes, then by the definition of \orthogoval{} they must be non-collinear in all other $s-1$ planes.  Thus there are at least $s-1$ rows with linearly independent entries in columns $z_1$, $z_2$ and $z_3$.
\end{proof}
The covering arrays constructed from these Sherwood covering perfect hash families can be extended by an additional two columns.
\begin{theorem}
  Suppose that a set of $s$ mutually \orthogoval{} Desarguesian affine planes exists.  Then there exists a $\CA_{s-1}(sq^3-q; 3, q^2+2 , q)$. 
  \end{theorem}
\begin{proof}
The affine plane embeds uniquely into the projective plane so we can extend the set of $s$ mutually \orthogoval{} affine planes to a set of $s$ projective planes which are \orthogoval{} {\em except} for the line $z$. From these we construct the $s \times (q^2+q+1)$ array $C$ as in the proof of \cref{thm_orthog2cphf} defined by
  \[
    C_{iz} = \phi_i(z).
  \]
Because the planes have the property that any three points not {\em all} on the line $z$ which are collinear in one are non-collinear in all the others, if we delete all but two of the columns of $C$ which are indexed by the points on the line $z$, the resulting array is a $\CPHF_{s-1}(s;3,q^2+2,q)$.

When constructing the covering array from this extended $\SCPHF$ we observe that the repeated rows arise from the multiplication of $r$ with the entry $z$ in $C$ where the vectors $r$ are the vectors with only the last position non-zero, see~\cite{MR3770276}.  The multiplication of these vectors $r$ with the two points kept from the line $z$ ($(1:0:0)$ and $(0:1:0)$ for example) yield all zeros so the rows are still repeated in the covering array extended by two columns and the appropriate number can still be deleted.
\end{proof}

From our constructions of sets of mutually \orthogoval{} Desarguesian projective planes in \cref{thm_pair_proj,thm_proj_sum} we can build the corresponding covering perfect hash families and covering arrays.
\begin{corollary}
  \begin{itemize}
  \item For all prime powers $q$ there exists a $\CA(2q^3-1;3,q^2+q+1,q)$.
  \item For $\lambda <4$ there exists a $\CA_{\lambda}(27\lambda+26;3,13,3)$.
    \end{itemize}
  \end{corollary}

From our constructions of sets of mutually \orthogoval{} Desarguesian affine planes summarized in \cref{orthog_ag} and \cref{thm_affine_sum} we can build the corresponding covering perfect hash families and covering arrays.
\begin{corollary}
  \begin{itemize}
  \item For all prime powers $q=2^n$ there exists a $\CA(2q^3-q;3,q^2+2,q)$.
  \item For all prime powers $q=2^n$ with $\gcd(n,6)=1$ and $\lambda < 3$, there exists a $\CA_{\lambda}(\lambda q^3+q^3-q;3,q^2+2,q)$.
  \item For $\lambda <7$ there exists a $\CA_{\lambda}(27\lambda+24;3,11,3)$.   
  \item For $\lambda <7$ there exists a $\CA_{\lambda}(64\lambda+60;3,18,4)$.
  \item For $\lambda <7$ there exists a $\CA_{\lambda}(512\lambda+504;3,66,8)$.
    \end{itemize}
  \end{corollary}

We believe these are probably the best known construction for higher index covering arrays in the instances where the sets of mutually \orthogoval{} planes exist.  

\section{Final thoughts}\label{sec_final}

Many open problems remain.  Foremost among them is to determine the largest size of a set of mutually \orthogoval{} planes of order $q$ for all prime powers $q$.  We have some exact values for small $q$ in Theorems \ref{thm_proj_sum} and \ref{thm_affine_sum}.  We have the upper bound, usually $q+2$, of \cref{thm_upperbound} and the lower bound of 2 from \cref{thm_pair_proj} and \cref{orthog_ag}, and for some $q$ we have a lower bound of 3 affine planes from \cref{cor_3_affine_planes}. There is a large gap between the lower and upper  bounds for all but the smallest $q$.

A secondary question is how many essentially different maximal sets of mutually \orthogoval{} planes of order $q$ exist.  Even in the case of the pairs or triples of mutually \orthogoval{} planes constructed in this paper, we do not know if the constructed sets are non-isomorphic.

It may be possible to use \cref{thm_pair_proj}  to find sets of more than two mutually \orthogoval{} planes. If we pick two different sets of three conjugate points, we can use each to build other planes on the point set of the standard plane that are each \orthogoval{} to the standard plane. The transformation from the second to third plane is just the composition of the two Cremona transformations that formed each.  If this composition is itself a conjugated Cremona then we may have constructed a set of three mutually \orthogoval{} projective planes.

An alternative approach to construct more than two \orthogoval{} projective planes is to use difference sets.  Baker \etal \cite{baker_projective_1994} prove that if  $D \subset \mathbb{Z}_{n^2+n+1}$ is a difference set, then $-D$, $2D$ and $D/2$ are all difference sets that are conics.  These could potentially be used to create a set of more than two mutually \orthogoval{} planes of orders other than 3. For example $-D/2 = 2D \neq D$ if and only if $-4$ is a multiplier of $D$ and 2 is not.  In this case, all three of $D$, $-D$ and $2D $ must be distinct and are each related to each other by multipliers of $-1$, $2$ or $1/2$ and thus form ovals in the planes generated by the other.  Gordon, Mills and Welch showed that the only multipliers of a Desarguesian planar difference set in $\mathbb{Z}_{q^2+q+1}$ for $q = p^e$ are the powers of $p$ in $\mathbb{Z}_{q^2+q+1}$ \cite{MR146135}.  Thus if there exists $q = p^e$ such that $p^d \equiv -4 \pmod{q^2+q+1}$ and 2 is not generated by $p$, then there exists a set of three Desarguesian \orthogoval{} projective planes of order $q$. Unfortunately, 3 is the only prime power up to 10,000,000 for which this is true.

It is also natural to generalize the definition of \orthogoval{} to other geometric objects.  Possibilities that are intriguing include inversive planes and higher dimensional projective and affine spaces.

\section{Acknowledgements}

We would like to thank Erin Lanus and Kaushik Sarkar. They found an $\SCPHF(2;3,8,64)$ in a conditional expectation search.  This turned out to have the structure of a pair of \orthogoval{} affine planes of order 8 and discussion of this structure in several venues was the impetus for the research presented in this article.  We thank Stefaan De Winter who pointed out that the set of seven mutually \orthogoval{} affine planes of order 4 constructs a Kirkman triple system of order 15.
      
\bibliographystyle{plain}
\bibliography{geobib}
\end{document}